\documentclass[11pt,reqno]{amsart}
\usepackage{amsmath, latexsym, amsfonts, amssymb,
amsthm, amscd,epsfig,enumerate}
\advance\hoffset -.75cm

\usepackage{amsfonts}
\usepackage{latexsym}
\usepackage{amsmath}
\usepackage{amssymb}

\newcommand{\C}{\mathbb C}
\newcommand{\R}{\mathbb R}

\newcommand{\N}{\mathbb N}

\newcommand{\eps}{\varepsilon}

\newcommand{\ident}{{\mathchoice {\rm 1\mskip-4mu l} {\rm 1\mskip-4mu l}
{\rm 1\mskip-4.5mu l} {\rm 1\mskip-5mu l}}}
\newcommand{\undertilde}[1]{\underset{\widetilde{}}{#1}}

\newtheorem{teo}{Theorem}[section]

\newtheorem{cor}[teo]{Corollary}
\newtheorem{rem}[teo]{Remark}
\newtheorem{pro}[teo]{Proposition}

\newtheorem{teo2}{Theorem}
\newtheorem{exm}[teo2]{Example}

\begin{document}

\title
{Characterization of critical values of \\ branching random walks on
weighted graphs through \\ infinite-type branching processes}

\author[D.~Bertacchi]{Daniela Bertacchi}
\address{D.~Bertacchi,  Universit\`a di Milano--Bicocca
Dipartimento di Matematica e Applicazioni,
Via Cozzi 53, 20125 Milano, Italy
}
\email{daniela.bertacchi\@@unimib.it}

\author[F.~Zucca]{Fabio Zucca}
\address{F.~Zucca, Dipartimento di Matematica,
Politecnico di Milano,
Piazza Leonardo da Vinci 32, 20133 Milano, Italy.}
\email{fabio.zucca\@@polimi.it}

\date{}

\begin{abstract}
We study the branching random walk on weighted graphs; site-breeding and
edge-breeding branching random walks on graphs are seen as particular cases.
Two kinds of survival can be identified: a weak survival (with positive probability there is at least one
particle alive somewhere at any time) and a strong survival (with positive probability the colony survives
by returning infinitely often to a fixed site). The behavior of the process
depends on the value of a certain parameter which controls the birth rates;
the threshold between survival and (almost sure) extinction is called critical value.
We describe the strong critical value in terms of a geometrical parameter of
the graph. We characterize the weak critical value and relate it to another
geometrical parameter. We prove that, at the strong critical value, the
process dies out locally almost surely; while, at the weak critical
value, global survival and global extinction are both possible.
\end{abstract}

\maketitle



\baselineskip .6 cm

\section{Introduction}\label{sec:intro}
\setcounter{equation}{0}

We consider
the branching random walk (briefly BRW) as a continuous-time process
where particles live on an at most countable set $X$ (the set of sites).
Each particle lives on a site and, independently of the others, has a random lifespan; during its life it breeds at random intervals and sends its
offspring to randomly chosen sites.
More precisely each particle has an exponentially distributed lifespan
with mean 1. To a particle living at site $x$, for any $y\in X$, there corresponds a Poisson clock of rate $\lambda k_{xy}$: when the clock
rings, a new particle is born in $y$ (where $(k_{xy})_{x,y\in X}$
is a matrix with nonnegative entries and $\lambda>0$), provided that the
particle at $x$ is still alive.

This approach unifies the two main points of view which may be found in
the literature: the \textit{site-breeding} BRW and the most widely used \textit{edge-breeding}
BRW. Indeed in the first case there is a constant reproduction rate $\lambda$
at each site and the offspring is sent accordingly to a probability
distribution on $X$ (thus $(k_{xy})_{x,y\in X}$ is a stochastic matrix).
Examples can be found in \cite{cf:BZ}, \cite{cf:HuLalley} and \cite{cf:Stacey03}
(where it is called \textit{modified} BRW).
In the edge-breeding model, $X$ is a graph and to each (oriented)
edge one associates a reproduction rate $\lambda$
(thus $(k_{xy})_{x,y\in X}$ is the adjacency matrix of the graph).
Some examples are in \cite{cf:BZ}, \cite{cf:Ligg2},  \cite{cf:Pem},
\cite{cf:PemStac1} and \cite{cf:Stacey03}.
On regular graphs (see for instance \cite{cf:Ligg1} and \cite{cf:MadrasSchi}), the site-breeding model employing the
transition matrix of the simple random walk is equivalent, up to a multiplicative constant, to the edge-breeding one.

We consider the BRW with initial configuration given by a single
particle at a fixed site $x$: there are two kinds of survival:
\begin{enumerate}[$(i)$]
\item
\textit{weak} (or global) \textit{survival} -- the total number of particles is positive at each time;
\item
\textit{strong} (or local) \textit{survival} -- the number of particles at  site $x$ is not eventually $0$.
\end{enumerate}
Let us denote by $\lambda_w(x)$ (resp.~$\lambda_s(x)$) the infimum of the values
of $\lambda$ such that there is weak (resp.~strong)
survival with positive probability. Clearly $\lambda_w(x) \le \lambda_s(x)$
and these values do not depend on $x$ when $K$ is irreducible
(see Section~\ref{subsec:graphs}).
For the edge-breeding BRW on a connected graph, in \cite{cf:PemStac1} it was proved
that $\lambda_s=1/M_s$ where $M_s$ is a geometrical parameter of the
graph. This result can be extended to the BRW on weighted graphs
(Theorem~\ref{th:pemantleimproved}). To our knowledge the
behavior of the BRW at $\lambda=\lambda_s(x)$ was yet unknown:
we prove that there is almost sure extinction in Theorem~\ref{th:critb}
(we proved the same result for BRW on multigraphs in \cite{cf:BZ}).

More challenging is the characterization of the weak critical parameter
$\lambda_w(x)$ and the study of the weak critical behavior.
Following the ideas which lead to the characterization of $\lambda_s(x)$
one naturally guesses that $\lambda_w(x)=1/M_w(x)$ (see Section~\ref{subsec:graphs} for the definition).
Indeed in \cite{cf:BZ} we proved that in the irreducible case, $\lambda_w\ge1/M_w$ and we gave sufficient conditions for equality
(for instance all site-breeding BRWs satisfy these conditions).
In this paper we use a different approach which allows us to characterize
 $\lambda_w(x)$ in terms of the existence of solutions of certain
infinite-dimensional linear systems (Theorem~\ref{th:equiv1});
in particular we show that $\lambda_w(x)$ is related to the
so-called Collatz-Wielandt numbers of some linear operator
(see \cite{cf:FN1}, \cite{cf:FN2} and \cite{cf:Marek1} for the
definition).

Thanks to this characterization, we prove a stronger lower bound,
$\lambda_w(x)\ge1/M_w^-(x)$ and give sufficient conditions for equality
(Remark~\ref{finite} and Propositions~\ref{th:fgraph} and \ref{th:condU}).
We show (Example~\ref{exm:2}) that it may be that $\lambda_w(x)=
1/M_w^-(x)\neq1/M_w(x)$. As for the critical behavior, Example~\ref{exm:4}
is a BRW which globally survives at $\lambda_w(x)$ (while for instance
on finite weighted graphs the BRW dies out at $\lambda_w(x)$ -
this is a particular case of Theorem~\ref{th:critg}).
The question whether $\lambda_w(x)=1/M_w^-(x)$ always holds
is, as far as we know, still open.

The basic idea behind the study of $\lambda_w(x)$ relies on the
comparison between the BRW and an infinite-type branching process
(briefly IBP).
It is well known that the probability of extinction of a
Galton-Watson branching process is the smallest positive fixed point
of a certain generating function.
In Section~\ref{sec:GBP} we prove some results on IBPs
by studying an infinite-dimensional generating function and its
fixed points.

The paper is organized as follows: Sections~\ref{subsec:graphs}
and \ref{subsec:genfun} introduce the basic definitions
(among which the definition of weighted graph and of the geometrical
parameters of the graph).
In Section~\ref{sec:fixed} we prove some results on fixed points
for monotone functions in partially ordered sets.
In Section~\ref{sec:GBP} we define IBPs and associate
in a ``canonical'' way an IBP to a given BRW.
Section~\ref{sec:critical} is devoted to the study of
the critical values $\lambda_s(x)$ and $\lambda_w(x)$
(Section~\ref{subsec:critical}) and of the strong and weak critical
behaviors (Section~\ref{subsec:criticalb}).
Finally in Section~\ref{sec:examples} we give some examples of IBPs and BRWs.

\section{Basic definitions and preliminaries}\label{sec:def}

\subsection{Weighted graphs}\label{subsec:graphs}

Let us consider $(X,K)$ where $X$ is a countable (or finite) set and $K=(k_{xy})_{x,y \in X}$ is
a matrix of nonnegative \textit{weights} (that is, $k_{xy} \ge 0$) such that $\sup_{x \in X} \sum_{y \in X}
k_{xy} = M< \infty$. We denote by $(X,K)$ the \textit{weighted graph} with set of edges
$E(X):=\{(x,y) \in X \times X: k_{xy}>0 \}$, where to each edge $(x,y)$ we associate
the weight $k_{xy}$.
We say that $K$ is \textit{irreducible} if
$(X,E(X))$ is a connected graph.

We define recursively
$k^n_{xy}:=\sum_{w\in X} k^{n-1}_{x w} k_{w y}$
(where $k^0_{xy}:=\delta_{xy}$); moreover we set
$T^n_x:=\sum_{y \in X} k^n_{xy}$ and
$\phi^n_{xy}:=\sum_{x_1,\ldots,x_{n-1} \in X \setminus\{y\}} k_{x x_1} k_{x_1 x_2} \cdots k_{x_{n-1} y}$;
by definition
$\phi^0_{xy}:=0$ for all $x,y \in X$.
Clearly $k^n_{xy}$ is the total weight of all paths of length $n$ from $x$ to $y$,
$T^n_x$ is the total weight  of all paths of length $n$ from $x$, while
$\phi^n_{xy}$ is the analog of $k^n_{xy}$ regarding only paths reaching
$y$ for the first time at the $n$-th step.

For $k^n_{xy}$ and $T_x^n$ the following recursive
relations hold for all $n,m \geq 0$
\[
k^{n+m}_{xy}=\sum_{w \in X} k^n_{xw} k^m_{wy}; \\
\qquad
\begin{cases}
T_x^{n+m}=\sum_{w \in X} k^n_{xw} T_w^m \\
\\
T_x^0=1\\
\end{cases}
\]
and, for all $n \ge 1$,
\[
k_{xy}^n=\sum_{i=0}^n \phi_{xy}^i k^{n-i}_{yy}.
\]
Whenever, given $x,y \in X$, there exists $n \in \N$ such that $k^n_{xy}>0$ we
 write $x \to y$; if $x \to y$ and $y \to x$ then
 we write $x \leftrightarrow y$.
This is an equivalence relation; let us denote by $[x]$ the equivalence class of $x$ (usually
called \textit{irreducible class}).
We observe that the summations involved in $k^n_{xx}$ could be equivalently restricted
to sites in $[x]$, moreover $\lambda_s(x)$ depends only on
$[x]$.
Similarly one can prove that $\lambda_w(x)$ depends only on $[x]$.

We introduce the following
geometrical parameters
\[
M_s(x,y;X):=\limsup_{n} (k^n_{xy})^{1/n}, \qquad
M_w(x;X):=\limsup_{n} (T_x^n)^{1/n}, \qquad
M^-_w(x;X):=\liminf_{n} (T_x^n)^{1/n}.
\]
In the rest of the paper, whenever there is no ambiguity, we will omit
the dependence on $X$.
Moreover, we write $M_s(x):=M_s(x,x)$;
supermultiplicative arguments imply that
$M_s(x) =\lim_{n} (k^{dn}_{xx})^{1/dn}$ for some $d \in \N$ hence,
for all $x \in X$, we have that
$M_s(x) \le M^-_w(x) \le M_w(x)$.
%
%
It is easy to show that the above
quantities are constant within an irreducible class; hence in the
irreducible case the dependence on $x,y$ will be omitted.

\subsection{Generating functions}\label{subsec:genfun}

Let us consider the following generating functions
\[
\begin{split}
\Gamma(x,y|\lambda)&:=\sum_{n =0}^\infty k^n_{xy} \lambda^n,
\qquad \Theta(x|\lambda):=\sum_{n =0}^\infty T_x^n \lambda^n,
\qquad \Phi(x,y|\lambda):=\sum_{n =1}^\infty \phi_{xy}^n \lambda^n;
\end{split}
\]
note that the radii of convergence of $\Gamma(x,y|\lambda)$ and $\Theta(x|\lambda)$ are
$1/M_s(x,y)$ and $1/M_w(x)$ respectively.
The following relation holds
\begin{equation}\label{eq:HTheta}
\Gamma(x,y|\lambda)=\Phi(x,y|\lambda)\Gamma(y,y|\lambda)+\delta_{xy}, \quad \forall
\lambda: |\lambda|< \min(1/M_s(x,y),1/M_s(y)).
\end{equation}
Since
\begin{equation}\label{eq:genfun1}
\Gamma(x,x|\lambda)=\frac{1}{1-\Phi(x,x|\lambda)},
\qquad \forall \lambda \in \C: |\lambda|< 1/M_s(x),
\end{equation}
we have that $1/M_s(x)=\max\{ \lambda \geq 0 :\Phi(x,x|\lambda)\leq 1\}$
for all $x \in X$ (see Section 2.2 of \cite{cf:BZ} for details).

\subsection{Fixed points in partially ordered sets}\label{sec:fixed}

Let $(Q, \ge)$ be a partially ordered set and $W:Q \mapsto Q$ be a nondecreasing function,
that is, $x\ge y$ implies $W(x)\ge W(y)$.
Let us denote by $(-\infty, y]$ and $[y, +\infty)$ the \textit{intervals} $\{w \in Q: w \le y\}$ and
$\{w \in Q: w \ge y\}$ respectively.
We consider a topology $\tau$ on $Q$ such that
all the intervals $(-\infty, y]$ and $[y, +\infty)$ are closed.
\begin{pro}
\label{teo:monotone}
Let $W:Q \mapsto Q$ be a nondecreasing function.
\begin{enumerate}
\item[(a)]
If $q \ge W(q)$ then $W((-\infty,q]) \subseteq (-\infty,q]$.
If $q \le W(q)$ then $W([q,+\infty)) \subseteq [q,+\infty)$.
\end{enumerate}
Moreover let us suppose that $q_0 \in Q$ satisfies $W(q_0) \ge q_0$
(resp.~$W(q_0) \le q_0$)
and define the sequence $\{q_n\}_{n \in \N}$
recursively by $q_{n+1}=W(q_n)$, for all $n \in \N$.
The following hold.
\begin{enumerate}
\item[(b)]
The sequence
is nondecreasing (resp.~nonincreasing).
\item[(c)]
If the sequence has a cluster point $q$
and $y$ is such that $y \ge q_0$, $y \ge W(y)$ (resp.~$y \le q_0$, $y \le W(y)$) then
$q \le y$ (resp.~$q \ge y$).
\item[(d)]
Every cluster point $q$  of
$\{q_n\}_{n \in \N}$ satisfies $q \ge q_0$ (resp.~$q \le q_0$).
If $W$ is continuous then
there is at most one cluster point $q$ and
\[
\begin{split}
W(q)&=q \quad \text{ and } \quad
(-\infty,q] =
\bigcap_{y \ge q_0:W(y) \le y} (-\infty, y]
=\bigcap_{y \ge q_0:W(y)=y} (-\infty, y] \\
\Big(\text{resp.~}W(q)&=q \quad \text{ and } \quad
(-\infty,q] =
\bigcup_{y \le q_0:W(y) \ge y} (-\infty, y]
=\bigcup_{y \le q_0:W(y)=y} (-\infty, y]\,\Big).
\end{split}
\]
\end{enumerate}
\end{pro}

\begin{proof}
\begin{enumerate}[(a)]
\item
Note that if $y \in (-\infty,q]$ then
$W(y) \le W(q) \le q$. The second assertion is proved analogously.
\item
This is easily proved by induction on $n$.
\item
By induction on $n$ we have $q_n \in (-\infty,y]$ which is closed
by assumption, thus $q \in (-\infty,y]$.  The second assertion is proved analogously.
\item
The first claim follows since $[q_0,+\infty)$ is closed.
Continuity implies that for every cluster point $W(q)=q$.
Moreover
if $q$ and $\widetilde q$ are two cluster points then since
$q_0\le \widetilde q$
then by  (c)
$q \le \widetilde q$ and similarly $\widetilde q \le q$ whence $q= \widetilde q$.
By (c)
$(-\infty, q]=\bigcap_{y \ge q_0:W(y) \le y} (-\infty, y]$.
Moreover since $W(q)=q$
\[
(-\infty, q] \supseteq
\bigcap_{y \ge q_0:W(y) = y} (-\infty, y] \supseteq
\bigcap_{y \ge q_0:W(y) \le y} (-\infty, y]
\]
whence the claim. The proof of the second claim is analogous.
\end{enumerate}

\end{proof}

\begin{cor}
\label{cor:monotone}
Let $Q$ have a smallest element $\mathbf 0$ (resp.~a largest element $\mathbf 1$),
$W:Q \mapsto Q$ be a continuous nondecreasing function.
If $\{q_n\}_{n \in \N}$ is recursively defined by
 \begin{equation}
 \label{eq:qn}
\begin{cases}
q_{n+1}=W(q_n) \\
q_0=\mathbf 0 \quad \text{ (resp.~}q_0=\mathbf 1 \text{)}. \\
\end{cases}
\end{equation}
then $\{q_n\}_{n \in \N}$ has at most one cluster point $q$;  moreover $q$
is the smallest (resp.~largest) fixed point of $W$ and for any $y \in Q$, we have that
$q<y$ (resp.~$q>y$) if and only if there exists $y^\prime < y$ (resp.~$y^\prime > y$) such that $W(y^\prime) \le y^\prime$
(resp.~$W(y^\prime) \ge y^\prime$).

\end{cor}

\begin{proof}
Clearly $W(\mathbf 0) \ge \mathbf 0$ hence (according to the previous proposition)
 the sequence
$\{q_n\}_{n \in \N}$ is nondecreasing and
since $q_0=\mathbf 0 \le y$ for all $y \in X$, there at most one cluster point
$q$, and it is the smallest
fixed point of $W$. If $q<y$ then take $y^\prime=q$; on the other hand if
there exists $y^\prime < y$ such that $W(y^\prime) \le y^\prime$ then $q\le y^\prime <y$.
The proof of the second claim follows analogously.
\end{proof}

\section{Infinite-type branching processes}\label{sec:GBP}

Let $X$ be a set which is at most countable. Each element of this set
represents a different type of particle of a (possibly) infinite-type branching process.
Given $f \in \Psi:=\{g \in \N^X:
S(g):=\sum_{x \in X} g(x)< +\infty\}$, at the end of its life
a particle of type $x$ gives birth to $f(y)$ children of type $y$ (for all $y \in X$) with
probability $\mu_x(f)$ where $\{\mu_x\}_{x \in X}$ is a family of probability
distributions on the (countable) measurable space $(\Psi,2^\Psi)$.

To the family $\{\mu_x\}_{x \in X}$ we associate a generating function $G:[0,1]^X \to [0,1]^X$
which can be considered as an infinite dimensional power series. More precisely,
for all $z \in [0,1]^X$ the function $G(z) \in [0,1]^X$ is defined as follows
\begin{equation}
\label{eq:genfun}
G(z|x):= \sum_{f \in \Psi} \mu_x(f) \prod_{y \in X} z(y)^{f(y)}.
\end{equation}
Note that  $G$ is continuous with respect to the \textit{pointwise convergence topology}
(or \textit{product topology}) on $[0,1]^X$.
Indeed,
every $f \in \Psi$ is finitely supported, hence
$\prod_{y \in X} z(y)^{f(y)}$ is
a finite product and
$z \mapsto \prod_{y \in X} z(y)^{f(y)}$
is continuous. 
The continuity of $G$ follows from Weierstrass criterion for uniform convergence,
since
$\sup_{z \in [0,1]^X}
\mu_x(f) \prod_{y \in X} z(y)^{f(y)}
= \mu_x(f)$ which is summable (with respect to $f \in \Psi$).

The set $[0,1]^X$ is partially ordered by $z \ge z^\prime$ if and only if $z(x) \ge z^\prime(x)$ for all $x \in X$;
by $z > z^\prime$ we mean that $z \ge z^\prime$ and $z \not = z^\prime$.
We denote by $\mathbf 0$ and $\mathbf 1$
the smallest and largest element of $[0,1]^X$ respectively, that is $\mathbf 0(x):=0$ and $\mathbf 1(x):=1$ for every $x \in X$.
The topological (partially ordered) space $[0,1]^X$ is compact and every monotone sequence has a cluster point, moreover
all the intervals $(-\infty,z] \equiv [\mathbf 0, z]$ and $[z,+\infty) \equiv [z,\mathbf 1]$
are closed sets whence all the hypotheses of Corollary~\ref{cor:monotone}
are satisfied. Let us note that $G(\mathbf 1)=\mathbf 1$
and  $G$ is  nondecreasing.
From now on we suppose that $\mu_x(\mathbf 0)>0$ for some $x \in X$
in order to avoid a trivial case of almost sure survival.

Let $q_n(x)$ be
the probability  of extinction before or at the $n$-th generation starting from a single
initial particle of type $x$; and let $q(x)$ be the probability of extinction
at any time starting from the same configuration.
Note that $q_n$ and $q$ can be viewed as
elements of $[0,1]^X$. Clearly $q_0=\mathbf 0$ and
\[\begin{split}
q_{n+1}(x) &= \sum_{f \in \Psi} \mu_x(f) \prod_{y \in X} q_n(y)^{f(y)}= G(q_n|x);\\
q(x)&=\lim_{n \to \infty} q_n(x).
\end{split}
\]
According to Proposition~\ref{teo:monotone} and Corollary~\ref{cor:monotone},
$q$ is the smallest fixed point of $G$ and $q< \mathbf 1$ if and only if
\begin{equation}
\label{eq:ineq}
G(y) \le y \qquad \text{for some } y < \mathbf 1.
\end{equation}
Hence,  if $y$ satisfies~\eqref{eq:ineq}
then $y(x)$ is an upper bound for $q(x)$.
Conversely if we define
\begin{equation}
\label{eq:H1}
H(v):=\mathbf1-G(\mathbf1-v)
\end{equation}
then $H$ is nondecreasing and continuous; moreover if
 \begin{equation}
 \label{eq:vn}
\begin{cases}
v_{n+1}=H(v_n) \\
v_0=\mathbf 1. \\
\end{cases}
\end{equation}
then $\{v_n\}_{n\in\N}$ is nonincreasing and has a unique cluster point
$v:=\lim_{n \to \infty} v_n=\mathbf1-q$.
Clearly $v_n(x)$ can be interpreted as the probability of survival up to
the $n$-th generation for the BRW starting with one particle
on $x$ ($v(x)$ being the probability of surviving forever).

Moreover $v > \mathbf 0$ if and only if $H(y) \ge y$ for
some $y \ge \mathbf 0$.
Note that in this case $y(x)$ is a lower bound for
$v(x)$.
Let $G_n$ and $H_n$ be the $n$-th iterates
of $G$ and $H$; $H_n(v)=\mathbf1-G_n(\mathbf1-v)$
and they are continuous and nondecreasing.

\begin{rem}
\label{rem:irrid}
Let us consider the graph $(X,E_\mu)$ where $E_\mu:=\{(x,y) \in X^2: \exists f \in \Psi, f(y)>0, \mu_x(f)>0\}$.
We call the IBP \textit{irreducible} if and only if the graph
$(X,E_\mu)$ is connected. It is easy to show that for the extinction probabilities $q$ of
an irreducible IBP we have
$q<\mathbf 1$ (that is $v>\mathbf 0$) if and only if $q(x)< 1$ for all $x \in X$
(that is $v(x)>0$ for all $x \in X$).
\end{rem}

\subsection{Infinite-type branching processes associated to branching random walks}
\label{subsec:ibrw}

In order to study the weak behavior of the BRW,
we associate a discrete-time branching process to the (continuous-time)
BRW in such a way that
they both survive or both die at the same time.
Each particle of the BRW living on a site $x$ will be given
the label $x$ which represents its type.
We suppose that the BRW starts from a single particle in
a vertex $x_0$; if there are several particles we repeat
this construction for each initial particle.
The IBP is constructed as follows: the 0th generation
is one particle of type $x_0$; the 1st generation of the IBP
is the collection of the children of this particle (ever born):
 this collection is almost surely
finite, say, $r_1$ particles in the vertex $x_1$, $\ldots$,
$r_m$ particles in $x_m$. Thus from the point of view of the IBP
the 1st generation is the collection of
$r_1$ particles of type $x_1$, $\ldots$,
$r_m$ particles of type $x_m$.
Take one particle of type $x_1$ in the 1-st generation and collect all its children,
repeat this
for all the particles in the 1st generation: the set of all
these new particles is the 2nd generation.
Proceeding
in the same way we construct the 3rd generation and so on.

Clearly the progeny of the IBP is the same as the progeny
of the BRW hence the latter is finite (i.e.~the BRW dies out)
if and only if the former is finite (i.e.~the IBP dies out).
The probabilities
of extinction of the IBP
(that is, the smallest fixed point of the
generating function), regarded as an element of $[0,1]^X$, coincide with
the probabilities of extinction
of the BRW.

Let us compute the generating function of this IBP.
Roughly speaking, the probability for a particle of type $x$ of having
$f(y)$ children of type $y$ for all $y \in X$ (where $f \in \Psi$)
is the probability that, for all $y \in Y$, a Poisson clock
of rate $\lambda k_{xy}$ rings $f(y)$ times before the death
of the original particle (i.e.~a clock of rate 1).
Elementary computations show that
\[
\mu_x(f)= \frac{
S(f)! \prod_{y \in X} (\lambda k_{xy})^{f(y)}}
{(1+\lambda \sum_{y \in X} k_{xy})^{
S(f)
+1}\prod_{y \in X} f(y)!}.
\]
Recalling~\eqref{eq:genfun} we have
\begin{equation}
\label{eq:G-BRW}
\begin{split}
G^\lambda(z|x)&=\sum_{f \in \Psi}
\frac{
S(f)! \prod_{y \in X} (\lambda k_{xy})^{f(y)}}
{(1+\lambda \sum_{y \in X} k_{xy})^{
S(f)+1}\prod_{y \in X} f(y)!}
 \prod_{y \in Y} z(y)^{f(y)} \\
&= \frac{1}{1+\lambda \sum_{y \in X} k_{xy}}
\sum_{i=0}^{+\infty} \sum_{f:
S(f)
=i}
\frac{i!}{\prod_{y \in X} f(y)!}
\frac{1}{
(1+\lambda \sum_{y \in X} k_{xy})^i
}
 \prod_{y \in X} \Big ( \lambda k_{xy}z(y)
 \Big )^{f(y)} \\
&=
\frac{1}{1+\lambda \sum_{y \in X} k_{xy}}
\sum_{i=0}^{+\infty}
\Big ( \frac{\lambda \sum_{y \in X} k_{xy}z(y) }
{1+\lambda \sum_{y \in X} k_{xy}} \Big )^i
 =
\frac{1}{1+\lambda \sum_{y \in X} k_{xy}(1-z(y))}.
\end{split}
\end{equation}
We note that
the quantity $\lambda k_{xy}$
can be interpreted as the expected number of offsprings of type $y$ of
a particle of type $x$.
Clearly in this case
\[
H^\lambda(v;x)=
\frac{\lambda \sum_{y \in X} k_{xy}v(y)}{1+\lambda \sum_{y \in X} k_{xy}v(y)}.
\]
If we define the bounded linear operator $K:l^\infty(X) \mapsto l^\infty(X)$ as
$Kv(x):=\sum_{y \in X} k_{xy}v(y)$ then
\begin{equation}
\label{eq:H-BRW}
H^\lambda(v)= \frac{\lambda Kv}{\mathbf1+\lambda Kv},
\end{equation}
hence the functions $H^\lambda_n$ and $G^\lambda_n$ from $[0,1]^X$ into itself are nondecreasing and continuous with respect
to $\|\cdot\|_\infty$ for every $n \ge 1$.
In particular each iterate $H^\lambda_n$ can be extended to the positive cone
$l^\infty_+(X):=\{v \in l^\infty(X): v \ge \mathbf 0\}$.
We observe that the operator $K$ preserves $l^\infty_+(X)$.

When there is no ambiguity, we will drop the dependence on $\lambda$ in these functions.
From now on, if not stated otherwise, it will be tacitly understood that
$G$ and $H$ are defined by equations~\eqref{eq:G-BRW} and~\eqref{eq:H-BRW}
respectively.

It is easy to show that $K$ is irreducible (as stated in
Section~\ref{subsec:graphs}) if and only if the corresponding
IBP is irreducible in the sense of Remark~\ref{rem:irrid}.

\section{The critical values and the critical behaviors}\label{sec:critical}

\subsection{The critical values}\label{subsec:critical}

In \cite{cf:PemStac1} it was proved that, in the irreducible case,
$\lambda_s=1/M_s$ for any
graph.
In \cite{cf:BZ} we used a different approach to extend this result to multigraphs;
the same arguments hold for weighted graphs (we repeat the proof for
completeness).
This approach allows us to study the critical behavior when $\lambda=\lambda_s(x)$
(see Theorem~\ref{th:critb}).
We observe that in the proof of the following theorem to the BRW we associate a particular branching process
which is not the one introduced in Section~\ref{subsec:ibrw}.
The proof relies on the concept of (reproduction) \textit{trail}: see \cite{cf:PemStac1}
for the definition.

\begin{teo}\label{th:pemantleimproved}
For every weighted graph $(X,K)$ we have that $\lambda_s(x)=1/M_s(x)$.
\end{teo}

\begin{proof}
Fix $x \in X$,
consider a path $\Pi:=\{x=x_0, x_1, \ldots, x_n=x\}$ and
define its number of cycles $\mathbb{L}(\Pi):=|\{i=1,
\ldots,n:x_i=x\}|$; the expected number of trails along such a path
is $\lambda^n \prod_{i=0}^{n-1} k_{x_i x_{i+1}}$
(i.e. the expected number of particles ever born at $x$, descending from the original particle
at $x$ and whose genealogy is described by the path $\Pi$ -- their mothers were at $x_{n-1}$, their
grandmothers at $x_{n-2}$ and so on).
Disregarding the original time scale, to the BRW there
corresponds a Galton-Watson branching process: given any particle $p$ in $x$
(corresponding to a trail with $n$ cycles), define its children
as all the particles whose trail is a prolongation of the trail of $p$ and is associated
with a spatial path with $n+1$ cycles.
Hence a particle is of the $k$-th generation if and only if the
corresponding trail has $k$ cycles; moreover it  has one (and only one)
parent in the $(k-1)$-th generation. Since each particle behaves
independently of the others then the process is markovian. Thus the BRW
survives strongly if and only if this branching process does.
The expected number of children of the branching process is the sum
over $n$ of the expected number of trails of length
$n$ and one cycle, that is $\sum_{n=1}^\infty \phi_{x,x}^n\lambda^n=\Phi(x,x|\lambda)$.
Thus we have a.s.~local extinction if and only if $\Phi(x,x|\lambda)\leq 1$, that is,
$\lambda(x) \leq 1/M_s(x)$.
\end{proof}

We turn our attention to the weak critical parameter $\lambda_w(x)$, which, by Corollary~\ref{cor:monotone}, may be characterized in terms of the function $H^\lambda$
(defined by equation~\eqref{eq:H-BRW}):
\begin{equation}
\label{eq:lambdaw}
\begin{split}
\lambda_w(x)&=\inf \{\lambda \in \R: \exists v \in l^\infty_+(X), v(x) >  0 , H^\lambda(v) \ge v\} \\
&=
\inf \Big \{\lambda \in \R: \exists v \in [0,1]^X, v(x)> 0 , \lambda Kv \ge \frac{v}{1-v} \Big \}.
\end{split}
\end{equation}
Our goal is to give other characterizations of $\lambda_w(x)$.
Theorem~\ref{th:equiv1} shows that, for every $n \ge 1$
\begin{equation}
\label{eq:lambdaw3}
\begin{split}
\lambda_w(x)&=\inf \{\lambda \in \R: \exists v \in l^\infty_+(X), v(x) >  0 , H^\lambda_n(v) \ge v\}\\
&=\inf \{\lambda \in \R: \exists v \in l^\infty_+(X), v(x)> 0, \lambda^n K^n v \ge v \};
\end{split}
\end{equation}
thus, by taking $n=1$ in the previous equation,
$\lambda_w(x)=\inf \{\/ \undertilde r\,\!_K(v): v \in l^\infty(X), v(x)=1\}$ where
$\undertilde r\,\!_K(v)$ is the
 \textit{lower
Collatz-Wielandt number of} $v$ (see \cite{cf:FN1}, \cite{cf:FN2} and \cite{cf:Marek1}).

We note that equation~\eqref{eq:lambdaw3} is particularly useful to compute
the value of $\lambda_w$ (indeed solving the linear inequality therein is easier than
solving the nonlinear inequality in~\eqref{eq:lambdaw}). Unfortunately
 the critical (global) survival of the BRW (with one initial particle at $x$)
 is equivalent to the existence of
 a solution of $\lambda_w(x) Kv\ge v/(1-v)$ with $v(x)>0$ (see Example~\ref{exm:4}).
The existence of a solution of $\lambda_w(x) Kv\ge v$ does not imply critical survival.

\begin{teo}
\label{th:equiv1}
Let $(X,K)$ be a weighted graph and let $x \in X$.
\begin{enumerate}[(a)]
\item If $\lambda \le \lambda_w(x)$ and $v \in [0,1]^X$ is such that $\lambda Kv \ge v/(1-v)$
then $\inf_{y:x \to y,v(y) > 0} v(y) = 0$.
\item For all $n \in \N$, $n \ge 1$ we have
\[
\lambda_w(x)
=\inf \{\lambda \in \R: \exists v \in l^\infty_+(X), v(x) >  0 \text{ such that } H_n^\lambda(v) \ge v\}.
\]
\item For all $n \in \N$, $n \ge 1$ we have
\[
\lambda_w(x)=
\inf \{\lambda \in \R: \exists v \in l^\infty_+(X), v(x)> 0 \text{ such that } \lambda^n K^n v \ge v \}.
\]
\end{enumerate}
\end{teo}

\begin{proof}
\begin{enumerate}[(a)]
\item
Let $X^\prime :=\{y \in X: x \to y\}$ and consider $v^\prime(y):=v(y) \ident_{X^\prime}(y)$, then
$\lambda K v^\prime \ge v^\prime/(1-v^\prime)$ (since $\lambda Kv^\prime(y)=\lambda Kv(y)$ for
all $y \in X^\prime$). Thus we may suppose, without loss of generality, that $X^\prime=X$.
For all $t \in [0,1]$ we have
$\lambda K(tv) \ge \frac{tv}{1-tv} \frac{1-tv}{1-v}$ and
$v \mapsto \frac{1-tv}{1-v}$ is nondecreasing.
By contradiction, suppose that $\inf_{y\in Y}v(y)=\delta>0$, hence
$\frac{1-tv}{1-v} \ge \frac{1-t\delta}{1-\delta} \mathbf 1$ and
$(\lambda\frac{1-\delta}{1-t\delta}) K(tv) \ge \frac{tv}{1-tv}$
thus, for all $t \in (0,1)$, $\lambda> \lambda\frac{1-\delta}{1-t\delta} \ge \lambda_w(x)$.
\item
Define $\lambda_n(x):=\inf \{\lambda \in \R: \exists v \in l^\infty_+(X), v(x) >  0 , H_n^\lambda(v) \ge v\}$.
Clearly $H^\lambda(v) \ge v$ implies $H_n^\lambda(v) \ge v$, thus $\lambda_w(x) \ge \lambda_n$.
If $\lambda >\lambda_n$ then by Corollary~\ref{cor:monotone}
the sequence $\{\widetilde v_i\}_{i\in\N}$ defined by $\widetilde v_0=\mathbf 1$,
$\widetilde v_{i+1}=H_n^\lambda(\widetilde v_i)$ converges monotonically to some $v > \mathbf 0$, namely
$\widetilde v_i \downarrow v$. But $\widetilde v_i=v_{ni}$ (for all $i \in \N$)
where the nonincreasing sequence $\{v_j\}_{j\in\N}$ is defined by
equation~\eqref{eq:vn}, whence $v_j \downarrow v$. By \eqref{eq:lambdaw},
since
$H^\lambda(v) =v$, we get $\lambda\ge\lambda_w(x)$.
\item
Define now $\lambda_n:=\inf \{\lambda \in \R: \exists v \in l^\infty_+(X), v(x)> 0 \text{ such that } \lambda^n K^n v \ge v \}$.
We prove that
$\lambda_{n} \le \lambda_w(x)$ for all $n\ge1$. Indeed, if $\lambda>\lambda_w(x)$,
then there exists $\tilde v$ such that $\lambda K\tilde v\ge\frac{\tilde v}{1-\tilde v}\ge\tilde v$.
By induction on $n$, $\lambda^nK^n\tilde v\ge\tilde v$,
thus, for all $n$, $\lambda\ge\lambda_n$ which implies $\lambda_n\le\lambda_w$.
On the other hand, if $\lambda >\lambda_n$ then there exists
$\lambda^\prime \in [\lambda_n, \lambda)$ such that
$(\lambda^\prime)^n K^n v \ge v$ for some $v \in l^\infty_+(X)$ such that $v(x)>0$.
If $\eps=\lambda/\lambda^\prime-1$
and $\delta>0$ is such that $\|\lambda K H^\lambda_{n-1}(\delta^\prime v)\|_\infty \le \eps$ for all $\delta^\prime \in (0, \delta]$
(which is possible since $H^\lambda_n$ is continuous and
$H^\lambda_n(\mathbf 0)=\mathbf 0$) then we have that
$H^\lambda_n(\delta^\prime v) \ge (\lambda/(1+\eps)) K H^\lambda_{n-1}(\delta^\prime v)$. By induction
on $n$ and since $K$ is a positive operator
there exists $\widetilde \delta>0$ such that
$H^\lambda_n(\widetilde \delta v) \ge (\lambda/(1+\eps))^n K^n H^\lambda_0(\widetilde \delta v)=(\lambda^\prime)^n K^n (\widetilde \delta v)
\ge \widetilde \delta v$
whence $\lambda \ge \lambda_w(x)$ by (b) and this implies $\lambda_n \ge \lambda_w(x)$.
\end{enumerate}
\end{proof}

\noindent The following theorem improves Lemma 3.2 of \cite{cf:BZ}.

\begin{teo}
\label{th:weak}
For every weighted graph $(X,K)$ we have that $\lambda_w(x) \ge 1/M^-_w(x)$.
\end{teo}

\begin{proof}
Let $\lambda<1/M^-_w(x)$.
If there exists $v\in l^\infty_+(X)$
such that $\lambda Kv \ge \frac{v}{1-v}\ge v$,
then for all $n\in\N$ we have $\lambda^n K^nv \ge v$.
Thus $\|v\|_\infty  \lambda^n \sum_{y \in X} k^n_{xy} \ge \lambda^n K^nv(x) \ge v(x)$,
but, since $\lambda \liminf_n \sqrt[n]{\sum_{y \in X} k^n_{xy}}<1$, we have that $\liminf_{n} \lambda^n \sum_{y \in X} k^n_{xy}=0$,
whence $v(x)=0$. By~\eqref{eq:lambdaw}, $\lambda\le\lambda_w(x)$.
\end{proof}

\begin{rem}
\label{finite}
Let us focus on the particular case where $X$ is finite. Clearly
if $K$ is irreducible, then
$\lambda_w=\lambda_s=1/M_w=1/M_s=1/M_w^-$
and these parameters do not depend
on the site of the initial particle.

If $X$ is finite, but $K$ is not irreducible,
it may happen that $\lambda_w(x)\neq\lambda_s(x)$
and also $\lambda_w(x)\neq\lambda_w(y)$
(although $\lambda_w(x)\leq\lambda_w(y)$ for all $y$ such that
$x\to y$).

Moreover in the finite case
$\lambda_w(x,[x])=1/M_w^-(x,[x])$ (where by adding $[x]$ we
consider the parameters corresponding to the process
restricted to $[x]$, namely $([x],K|_{[x]\times [x]})$):
the proof is the same as in Proposition 2.2
of \cite{cf:BZ}.
From this and the fact that the BRW starting from
one particle in $x$ survives globally if and only if it survives  (locally and globally) in at least one irreducible class,
it follows that $
 \lambda(x,X)=\min\{1/M_w^-(y,[y]):x \to y \}$.
By induction on the number of equivalence classes, it is not difficult to
prove that $M_w^-(x,X)= \max\{M_w^-(y,[y]):x \to y \}$
which proves, for finite weighted graphs, that
$\lambda_w(x,X)=1/M_w^-(x)$.

As for the critical behavior,
 the $\lambda_w(x)$-BRW dies out (globally, thus locally)
almost surely. Indeed if
$\lambda_w(x)<\lambda_w(y,[y])=\lambda_s(y)$ it cannot
survive confined to $[y]$. If $\lambda_w(x)=\lambda_w(y,[y])$ then
according to \textit{(a)} of Theorem~\ref{th:equiv1} the
probabilities of survival $v$ for the process confined to $[y]$
satisfy $\inf_{z \in [y]} v(z)=0$. Being $[y]$ finite and irreducible,
this means that $v(z)=0$ for all $z \in [y]$.
\end{rem}

We say that $(X,K)$ is \textit{locally isomorphic} to $(Y,\widetilde K)$ if and only if
there exists a surjective map $f:X \to Y$ such that
$\sum_{z \in f^{-1}(y)} k_{xz}=\widetilde k_{f(x)y}$
for all
$x \in X$ and $y \in Y$.
An $(X,K)$ which is locally isomorphic to some
$(Y,\widetilde K)$ ``inherits'' its $M_w^-$s and $\lambda_w$s
(in a sense which is clear in the proof of the following
proposition).
\begin{pro}\label{th:fgraph}
If $Y$ is a finite set and $(X,K)$ is locally
isomorphic to $(Y, \widetilde K)$
then $\lambda_w(x)=1/M_w^-(x)$.
\end{pro}

\begin{proof}
The definition of the map $f$ immediately implies that $\sum_{z \in X} k^n_{xz}
= \sum_{y \in Y} \widetilde k^n_{f(x)y}$, hence $M_w^-(x,X)=
 M_w^-(f(x),Y)$. Moreover
$\lambda_w(x,X)= \lambda_w(f(x),Y)$. Indeed it is easy to prove that
$\lambda_w(x,X) \ge \lambda_w(f(x),Y)$. Conversely, if $\lambda>\lambda_w(f(x),Y)$
and $\widetilde v$ is such that $\lambda \widetilde K \widetilde v \ge \widetilde v$
then we define $v(x):=\widetilde v(f(x))$. Clearly $\widetilde K \widetilde v (f(x))=
Kv(x)$ hence $\lambda \ge \lambda_w(x,X)$. Remark~\ref{finite} yields the conclusion.
\end{proof}
Examples of BRWs $(X,K)$ which are locally isomorphic to some finite $(Y,\widetilde K)$ are BRWs where $\sum_{z\in X}k_{xz}$
does not depend on $x$ (in this case $Y=\{y\}$ is a singleton and
$\widetilde k_{yy}=\sum_{z\in X}k_{xz}$).
Another example is given
by \textit{quasi-transitive} BRWs, that is, there exists a finite $X_0\subset X$ such that for any $x\in X$ there is
a bijective map $\gamma_x:X\to X$ satisfying $\gamma_x^{-1}(x)
\in X_0$ and $k_{yz}=k_{\gamma_x y\,\gamma_x z}$ for all
$y,z$ (in this case $Y=X_0$ and
$\widetilde k_{wz}=\sum_{y:y=\gamma_y(z)}k_{wy}$).

Let us consider now the irreducible case;
since $\lambda_w(x)$ and $M_w^-(x)$ do not depend on $x$ let us write $\lambda_w$ and $M_w^-$ instead.
Note that the characterization of $\lambda_w$ can be written as
\begin{equation}
\label{eq:lambdaw2}
\begin{split}
\lambda_w&=\inf \{\lambda \in \R: \exists v \in l^\infty(X), v >  \mathbf 0 , H^\lambda(v) \ge v\} \\
&=
\inf \Big \{\lambda \in \R: \exists v \in l^\infty(X), v> \mathbf 0 , \lambda^n K^nv \ge v \Big \}\\
&=\inf \{\lambda \in \R: \exists v \in l^\infty(X), v >  \mathbf0 , H^\lambda_n(v) \ge v\},
\end{split}
\end{equation}
where the requirement $v >  \mathbf0$ seems less restrictive
than $v(x)>0$ for all $x\in X$, which is the one we would
expect in view of equation~\eqref{eq:lambdaw}. Nevertheless
by Remark ~\ref{rem:irrid} it follows that if there exists
$v >  \mathbf0$ satisfying one of the inequalities in
\eqref{eq:lambdaw}, then there exists a solution
$v^\prime$ of the same inequality with $v^\prime(x)>0$ for all $x\in X$.

\begin{pro}\label{th:condU}
Let $(X,K)$ be an irreducible weighted graph.
If for all $\eps>0$ there exists $N$ such
that $\sum_{y \in X} k^{N}_{xy} \ge (M_w^--\eps)^{N}$,
 for all $x\in X$, then $\lambda_w =1/M_w^-$.
\end{pro}

\begin{proof}
Let $\lambda>1/M_w^-$. Choose $\eps$ such that
$\lambda(M_w^--\eps)>1$. Then $\lambda^NK^N\mathbf1(x)=\lambda^N
\sum_{y\in X}k^N_{xy}\ge (\lambda(M_w^--\eps))^N>1$. Hence by
Theorem~\ref{th:equiv1} $\lambda>\lambda_w$.
Theorem~\ref{th:weak} yields the conclusion.
\end{proof}
We note that if $(X,K)$ is irreducible and satisfies the hypothesis of Proposition~\ref{th:fgraph},
then Proposition~\ref{th:condU} provides an alternative proof of $\lambda_w=1/M_w^-$.
For an example (which is not locally isomorphic to a finite weighted graph),
where one can use Proposition~\ref{th:condU}, see Example 3 in \cite{cf:BZ}
(which is a BRW on a particular radial tree).

\subsection{The critical behavior}
\label{subsec:criticalb}

\begin{teo}\label{th:critb}
For each weighted graph $(X,K)$ if
$\lambda=\lambda_s(x)$ then the $\lambda$-BRW starting from one particle at $x \in X$ dies out locally almost surely.
\end{teo}

\begin{proof}
Recall that (see the proof of Theorem~\ref{th:pemantleimproved})
the $\lambda(x)$-BRW survives locally if and only if the Galton-Watson branching process
with expected number of children $\Phi(x,x|\lambda)$ does. Since
$\Phi(x,x|1/M_s)\leq 1$ and $\lambda_s(x)=1/M_s$ then there is
a.s.~local extinction at $\lambda_s(x)$.
\end{proof}

\begin{teo}\label{th:critg}
If $Y$ is a finite set and $(X,K)$ is locally isomorphic to $(Y, \widetilde K)$
then the $\lambda_w(x)$-BRW starting from one particle at $x \in X$ dies out
globally almost surely.
\end{teo}
\begin{proof}
By reasoning as in Proposition~3.7 of \cite{cf:BZ} it is clear that the $\lambda_w(x)$-BRW
on $X$ dies out if and only if the $\lambda_w(x)$-BRW
on $Y$ does.
Remark~\ref{finite} yields the conclusion.
\end{proof}

\section{Examples}
\label{sec:examples}

We start by giving an example of an irreducible IBP
where, although the expected number of children of each particle is less than 1,
nevertheless the colony survives with positive probability.

\begin{exm}
\label{exm:1}
Let $X=\N$, $\{p_n\}_{n\in\N}$ be a sequence in $[0,1)$ and suppose that a particle of type $n\ge1$ at the end of its life
has one child of type $n+1$ with
probability $1-p_n$, one child of type $n-1$ with probability $p_n/2$
(if $n=0$ then it has one child of type $0$ with probability $p_0/2$)
and no children with probability $p_n/2$.
The generating function $G$ can be explicitly computed
\[
G(z|n)=
\begin{cases}
\frac{p_n}{2} + \frac{p_n}{2} z(n-1)+ (1-p_n) z(n+1) & n \ge 1 \\
\frac{p_0}{2} + \frac{p_0}{2} z(0) + (1-p_0) z(1) & n=0.\\
\end{cases}
\]
This process clearly dominates the (reducible) one where
a particle of type $n$ at the end of its life
has one child of type $n+1$ with
probability $1-p_n$ and no children with probability $p_n$.
The latter process has generating function
$\widetilde G(z|n)=p_n + (1-p_n)z(n+1) $.
By induction it is easy to show that the probabilities of extinction before or at generation
$n$ of the second process
are $q_{n}(j)=1-\prod_{i=j}^{j+n-1} (1-p_i)$ for all $n \ge 1$; hence it survives with positive probability,
that is $q_n(0) \not\to 1$ as $n \to \infty$, if
and only if $\sum_{i=1}^\infty p_i < +\infty$.
\end{exm}

In all the following examples, $X=\N$ and $k_{ij}=0$ whenever $|i-j|>1$. Although
this looks quite restrictive, one quickly realizes that many BRWs are
locally isomorphic to BRWs of this kind. For instance,
every BRW on a homogeneous tree of degree $m$ (with $k_{ij}=1$ on each edge)
is locally isomorphic to the BRW on $\N$ with $k_{0\,1}:=m$, $k_{n\, n+1}:=m-1$,
$k_{n\, n-1}:=1$ and $0$ otherwise. More generally any \textit{radial BRW on a radial tree} is
locally isomorphic to a BRW on $\N$. Indeed a general radial BRW on a radial tree is constructed
as follows: let us consider two positive real sequences $\{k_n^+\}_{n\in\N}$,
$\{k_n^-\}_{n\in\N}$ and a positive integer valued sequence $\{a_n\}_{n\in\N}$. By construction, the root
of the tree is some vertex $o$ which has $a_0$ neighbors and the rates are $k_{o x}:=k_0^+$,
$k_{x o}:=k_0^-$ for all neighbors $x$. Each vertex $x$ at distance $1$ from $o$ has
$1+a_1$ neighbors (one is $o$) and
we set $k_{x y}:=k^+_1$ and $k_{y x}:=k^-_1$ for all its $a_1$ neighbors $y$ at distance $2$
from $o$. Now each vertex at distance $2$ from $o$ has $1+a_2$ neighbors,
an outward rate $k_2^+$ and an inward rate $k_2^-$ and so on. This BRW is clearly locally
isomorphic to (therefore it has the same global behavior of) the BRW on $\N$ with
$k_{n\, n+1}:=a_n k_n^+$, $k_{n+1\, n}:= k_n^-$ and $0$ otherwise.

The next one is an example of a BRW  on $\N$ which is not irreducible and where $\lambda_w>1/M_w$.
This answers
an open question raised in \cite{cf:BZ}.

\begin{exm}
\label{exm:2}
Let $\{k_n\}_{n \in \N}$ be a bounded sequence of positive real numbers and
let us consider the BRW on
$\N$ with rates $k_{ij}:=k_i$ if $i=j-1$ and $0$ otherwise.
By using Equations~\eqref{eq:vn} and~\eqref{eq:H-BRW} one can show
that $v_n(i)=\lambda^n \beta_{i+n} /(1+\sum_{r=1}^{n} \lambda^r
\beta_{i+n}/\beta_{i+n-r})$ where $\beta_n:=\prod_{i=0}^{n-1} k_i$.

In order to prove that
$\lambda_w(i)=1/\liminf_n \sqrt[n]{\beta_{n+i}/\beta_i}=1/M_w^-(i)$ (which
 does not depend on $i$, though the BRW is not irreducible)
one may either study the behavior of $\{v_n\}_{n\in\N}$ above or, which is
simpler, use Theorem~\ref{th:equiv1}.
Indeed, without loss of generality,
we just need to prove that
for all $\lambda > 1/\liminf_n \sqrt[n]{\beta_n}$ it is
possible to solve the inequality
$\lambda K v \ge v$ for some $v \in l^\infty(X)$, $v> \mathbf 0$.
One can easily check that $ v(n):=1/(\lambda^n \beta_n)$ is a solution;
since $\lambda > 1/\liminf_n \sqrt[n]{\beta_n}$ we have that
$\lim_n v(n) =0$ and then $v \in l^\infty(X)$
.

Note that $\lambda_w=1/M_w^-$ which may be different from $1/M_w$
with the following choice of the rates.
Our goal is to define big intervals of consecutive vertices where $k_{i\, i+1}=1$, followed
by bigger intervals of vertices where $k_{i\, i+1}=2$ and so on. The result is a BRW where
$M_w=\limsup_n \sqrt[n]{\beta_n}=2$ while
$M_w^- = \liminf_n \sqrt[n] \beta_n=1$.

Define $a_n:=\lceil \log 2/\log (1+1/n) \rceil$, $b_n:=\lceil \log 2/(\log 2 - \log (2-1/n)) \rceil$ and
$\{c_n\}_{n\ge1}$ recursively by $c_1=1$, $c_{2r}=a_{2r}c_{2r-1}$, $c_{2r+1}=b_{2r+1}c_{2r}$ (for all $r \ge 1$).
Let $k_{i }$ be equal to $1$ if $i \in (c_{2r-1}, c_{2r}]$ (for some $r \in \N$) and
equal to $2$ if $ i \in (c_{2r}, c_{2r+1}]$ (for some $r \in \N$).
Clearly $\sqrt[n]{\beta_n} \in [1,2]$ for all $n \in \N$ and it is easy to check that, for all $r \ge 1$,
$\sqrt[c_{2r+1}]{\beta_{c_{2r+1}}} > 2-1/(2r+1)$ and $\sqrt[c_{2r}]{\beta_{c_{2r}}} \le 1+1/(2r)$, whence
$2=\limsup_n \sqrt[n] \beta_n>\liminf_n \sqrt[n] \beta_n=1$.
\end{exm}

Although this BRW is not irreducible, it is clear that a slight modification
(that is, adding a small backward rate as in the following example) does not
modify significantly the behavior of the process and allows to construct an irreducible
example with the same property.
Finally, the last example shows that the weak critical
survival is possible (while, according to Theorem~\ref{th:critb}, any strong critical BRW
dies out locally).

\begin{exm}
\label{exm:4}
Let $X:=\N$ and $K$ be defined by
$k_{0\, 1}:=2$, $k_{n\, n+1}:=(1+1/n)^2$, $k_{n+1\, n}:= 1/3^{n+1}$ and $0$ otherwise.
Hence the inequality $\lambda Kv \ge v/(1-v)$ becomes
\[
\begin{cases}
2 \lambda v(1) \ge v(0)/(1-v(0)) \\
\lambda(v(n+1) (1+1/n)^2+v(n-1)/3^n) \ge v(n)/(1-v(n)).\\
\end{cases}
\]
Clearly $v(0)=1/2$ and $v(n):=1/(n+1)$ (for all $n \ge 1$) is a solution for all
$\lambda \ge 1$.
If $\lambda <1$ then one can prove by induction that a solution must satisfy
$v(n+1)/v(n) \ge \frac1\lambda\left(\frac{n}{n+1}\right)^2\left(
1-\frac{1}{2^n}\right)$ for all $n \ge 2$. Thus
$v(n+1)/v(n)$ is eventually larger than $1+\eps$ for some $\eps>0$,
hence either $v=\mathbf 0$ or $\lim_n v(n)= +\infty$.
This implies that $\lambda_w=1$ and there is global survival if $\lambda=\lambda_w$.

\end{exm}

\end{document}